\documentclass{amsart}
\usepackage{amssymb}
\usepackage{bbm}
\usepackage{amsfonts}
\usepackage{amsfonts}
\usepackage{amsfonts}
\usepackage{amscd}
\usepackage{mathrsfs}
\usepackage{amsmath}
\usepackage{enumitem}
\newtheorem{theorem}{Theorem}[section]

\newtheorem{proposition}{Proposition}[section]

\theoremstyle{definition}
\newtheorem{definition}[theorem]{Definition}
\newtheorem{example}[theorem]{Example}

\theoremstyle{remark}
\newtheorem{remark}[theorem]{Remark}
\numberwithin{equation}{section}

\newcommand{\blankbox}[2]

\begin{document}
\title{ Finite irreducible conformal modules of  rank two Lie conformal algebras}

\author{ Maosen Xu,    Yanyong Hong* and  Zhixiang Wu }
\address{School of Mathematics Sciences, Zhejiang University, Hangzhou, 310027, P.R.China} \email{390596169@qq.com}
\address{Department of Mathematics, Hangzhou Normal University,
Hangzhou, 311121, P.R.China} \email{hongyanyong2008@yahoo.com}
\address{School of Mathematics Sciences, Zhejiang University, Hangzhou, 310027, P.R.China} \email{wzx@zju.edu.cn}
\thanks{This work was supported by the National Natural Science Foundation of China (No. 11871421), the Zhejiang Provincial Natural Science Foundation of China (No. LY17A010015, LY20A010022) and the Scientific Research Foundation of Hangzhou Normal University (No. 2019QDL012).}
\begin{abstract}In the present paper, we prove that any  finite non-trivial  irreducible module over a rank two Lie conformal algebra $\mathcal{H}$ is of rank one. We also describe the actions of $\mathcal{H}$ on its finite  irreducible modules explicitly. Moreover, we show that all finite non-trivial  irreducible modules of  finite Lie conformal algebras whose semisimple quotient is the Virasoro Lie conformal algebra are of rank one.
\end{abstract}
\maketitle

\section{introduction}
The notion of  a Lie conformal algebra was introduced in \cite{K} providing an axiomatic description of
properties of the operator product expansion in conformal field theory. There are many other fields closely related to Lie conformal algebras such as vertex algebras, linearly compact Lie algebras and integrable systems. In the view of \cite{BDK}, the notion of a Lie conformal algebra is a generalization of that of classical Lie  algebra, i.e. it is just the Lie algebra over a pseudo-tensor category. Thus, Lie conformal algebras have many
``conformal analogue" notions and  properties as ordinary Lie algebras such as semisimple  Lie conformal algebras, cohomology groups of  Lie conformal algebras, conformal modules. It was proved in {\cite{DK}} that, up to isomorphism, every finite semisimple Lie conformal algebra is a direct sum of Lie conformal algebras of  the form:
                      $Vir, \ Vir \ltimes \text{Cur}(\mathfrak{g}), \text{Cur}(\mathfrak{g}) $
where $\mathfrak{g}$ is a finite dimensional semisimple Lie algebra. The cohomology theory of Lie conformal algebras was also established in \cite{BKV}. As ordinary Lie algebras, the second   cohomolology group  of  a Lie conformal algebra calculates  all abelian  extensions of this Lie conformal algebra by its modules (coefficients of the cohomology group). The rank one and two Lie conformal algebras were classified in   \cite{DK} and \cite{K1,HF,BCH,W1} respectively. The rank three Lie conformal algebras were partially classified in \cite{W}.
 The problem of classifying  all finite irreducible conformal modules of  a given Lie conformal algebra was also studied by many authors. Finite irreducible conformal modules of a semisimple Lie conformal algebra were classified in \cite{CK}. Finite irreducible conformal modules of other important Lie conformal algebras were investigated in \cite{WY, LHW, SXY} and so on. In the present paper, we prove that any  finite non-trivial  irreducible module over a rank two Lie conformal algebra $\mathcal{H}$ is of rank one. We also describe the actions of $\mathcal{H}$ on its finite  irreducible modules explicitly. The abelian extensions of some free rank two Lie conformal algebras  by their rank one modules are also computed. Moreover, we show that all finite non-trivial  irreducible modules of  finite Lie conformal algebras whose quotient is the Virasoro Lie conformal algebra are of rank one.

    The paper is organized as follows. In Section 2, we recall some basic notions and properties of Lie conformal algebras. In Section 3,   we study  finite  non-trivial irreducible modules  of free rank two Lie conformal algebras explicitly. As an application, we prove that any finite  non-trivial irreducible modules of finite Lie conformal algebras whose semisimple quotient is the Virasoro Lie conformal algebra must be free of rank one. In addition, the abelian extensions of some free rank two Lie conformal algebras  by their rank one modules are also computed.

    Through this paper, denote $\mathbb{C}$ and $\mathbb{Z}_+$ the sets of all complex numbers and all nonnegative integers respectively.  In addition, all vector spaces and tensor products are over $\mathbb{C}$. For any vector space $V$,
    we use $V[\lambda]$ to denote the set of polynomials of $\lambda$ with coefficients in $V$.

 \section{Preliminary}
 In this section, we recall some basic definitions and results of Lie conformal algebras and their modules.

 \begin{definition}\label{d1}  A {\it Lie conformal algebra} $\mathcal{A}$ is a $\mathbb{C}[\partial]$-module together  with a $\mathbb{C}$-linear map (call $\lambda$-bracket)  $\mathcal{A} \otimes \mathcal{A} \rightarrow \mathcal{A}[\lambda]$, $a\otimes b \mapsto [a_\lambda b]$, satisfying the following axioms
\begin{equation*}
\begin{aligned}
&[\partial a_\lambda b]=-\lambda[a_\lambda b], \   [a_\lambda \partial b]=(\partial+\lambda)[a_\lambda b],\ \ (conformal\   sequilinearity),\\
&[a_\lambda b]=-[b_{-\lambda-\partial}a]\ \ (skew-symmetry),\\
&[a_\lambda [b_\mu c]]=[[a_\lambda b]_{\lambda+\mu}c]+[b_\mu[a_\lambda c]]\ \ (Jacobi \ identity),
\end{aligned}
\end{equation*}
for all $a$, $b$, $c\in \mathcal{A}$.
\end{definition}
By means of  conformal  sequilinearity, we can define a Lie conformal algebra by giving  the $\lambda$-brackets on its generators over $\mathbb{C}[\partial]$. In addition, the \emph{rank} of a Lie conformal algebra is  its rank as a $\mathbb{C}[\partial]$-module. We say a Lie conformal algebra {\it finite} if it is finitely generated as a $\mathbb{C}[\partial]$-module.
\begin{example} (\cite{DK})
The Virasoro Lie conformal algebra  $Vir=\mathbb{C}[\partial]L$ is a free  $\mathbb{C}[\partial]$-module of rank one, whose $\lambda$-brackets are determined by  $[L_\lambda L]=(\partial+2\lambda)L$. Further, it is well known that $Vir$ is a simple Lie conformal algebra.

For a Lie algebra $\mathfrak{g}$, the current Lie conformal algebra $\text{Cur}\mathfrak{g}$ is a free $\mathbb{C}[\partial]$-module $\mathbb{C}[\partial] \otimes \mathfrak{g}$  equipped with the $\lambda$-brackets:
     \[ [a_\lambda b]=[a,b],\ \ \text{for}\ \  a,b \in \mathfrak{g}. \]
  The Lie conformal algebra $Vir \ltimes \text{Cur}{\mathfrak{g}}$ is  defined by the following $\lambda$-brackets
    \[  [L_\lambda a]=(\partial+\lambda)a,\qquad  \text{ for all }a\in\mathfrak{g}.\]
\end{example}
 \begin{example}
    For $a,b \in \mathbb{C}$, the Lie conformal algebra $\mathcal{W}(a,b)=\mathbb{C}[\partial]L \oplus \mathbb{C}[\partial]Y$ is a free  rank two $\mathbb{C}[\partial]$-module with:
\[ [L_\lambda Y]=(\partial+a\lambda+b)Y, \  [L_\lambda L]=(\partial+2\lambda)L, \ [Y_\lambda Y]=0. \]
\end{example}
Suppose $\mathcal{A}$ is a Lie  conformal algebra.  For any $a,b \in \mathcal{A}$, we write
     \begin{eqnarray}{\label{e2}}
     \begin{array}{ll}
     [a_\lambda b]=\sum_{j\in \mathbb{Z}_+}(a_{(j)}b)\frac{\lambda^j}{j!}.
     \end{array}
     \end{eqnarray}
For every $j\in \mathbb{Z}_+$, we have the $\mathbb{C}$-linear map: $\mathcal{A}\otimes  \mathcal{A} \rightarrow  \mathcal{A}$, $a\otimes b \mapsto a_{(j)}b$, which is called {\it $j$-th product}.
 A $\mathbb{C}[\partial]$-submodule $\mathcal{B}$  of $\mathcal{A}$ is called a  \emph{subalgebra} of $\mathcal{A}$ if $\mathcal{B}$ is closed under all  $j$-th products , that is , $a_{(j)}b \in \mathcal{B}$ for any $a, b \in \mathcal{B}$ and $j \in \mathbb{Z}_+$. A subalgebra $\mathcal{B}$ is called an \emph{ideal} of $\mathcal{A}$ if  $a_{(j)}b \in \mathcal{B}$ for any $a \in \mathcal{A}$, $b \in \mathcal{B}$,  $j \in \mathbb{Z}_+$.   For example,  $\text{Cur}\mathfrak{g}$ is an ideal of $Vir \ltimes \text{Cur}\mathfrak{g}$, and $Vir$ is a subalgebra of  $Vir \ltimes \text{Cur}\mathfrak{g}$.

The derived algebra of a Lie conformal algebra $\mathcal{A}$ is the  vector space $\mathcal{A}'=\text{span}_{\mathbb{C}}\{a_{(n)} b\mid a\in \mathcal{A}, b \in \mathcal{A},  n\in \mathbb{Z}_+\}$. It is easy to check that  $\mathcal{A}'$ is an ideal of  $\mathcal{A}$. Define $\mathcal{A}^{(1)}=\mathcal{A}'$ and $\mathcal{A}^{(n+1)}={\mathcal{A}^{(n)}}'$ for any $n \geq 1$. Then we get the  {\it derived   series}
\begin{eqnarray*}
\mathcal{A}= \mathcal{A}^{(0)}\supset\mathcal{A}^{(1)}\supset\cdots  \mathcal{A}^{(n)}\supset\mathcal{A}^{(n+1)}\supset\cdots
\end{eqnarray*}
       of $\mathcal{A}$.
  If there exists $N \in \mathbb{Z}_+$ such that $\mathcal{A}^{(n)}=0$ for any $n\geq N$, then
$\mathcal{A}$ is said to be {\it solvable}. Note that the sum of two solvable ideals of a Lie conformal algebra $A$ is still a solvable ideal. Thus if $M$ is a finite Lie conformal algebra, the maximal solvable ideal exists and is unique which we denote by $\text{Rad}(M)$. A Lie conformal algebra  is {\it semisimple} if its radical is zero. Thus $M/\text{Rad}(M)$ is semisimple for any finite Lie conformal algebra $M$.

It was proved in {\cite{DK}} that any finite semisimple Lie conformal algebra is the direct sum of   the following Lie conformal algebras:
                      \[ Vir,\  \text{Cur}(\mathfrak{g}), \ Vir \ltimes \text{Cur}(\mathfrak{g}),  \]
     where $\mathfrak{g}$ is a finite dimensional semisimple  Lie algebra.

 Next, let us recall the definition of conformal module of a Lie conformal algebra $\mathcal{A}$.

 \begin{definition}{\label{d2}} For a Lie conformal algebra $\mathcal{A}$, a {\it conformal $\mathcal{A}$-module} $M$  is a $\mathbb{C}[\partial]$-module with a $\mathbb{C}$-linear map  $\mathcal{A} \otimes M \rightarrow M[\lambda]$, $a\otimes m \mapsto a_\lambda m$, satisfying the following axioms:
\begin{equation*}
\begin{aligned}
&\partial a_\lambda m=-\lambda(a_\lambda m), \ \   a_\lambda \partial m=(\partial+\lambda)a_\lambda m,\\
&a_\lambda (b_\mu m)=[a_\lambda b]_{\lambda+\mu}m+b_\mu(a_\lambda m),
\end{aligned}
\end{equation*}
for all $a$, $b\in \mathcal{A}$ and $m\in M$.
\end{definition}
Since  we only consider   conformal modules   in this paper,  the  conformal $\mathcal{A}$-modules are simply called  {\it $\mathcal{A}$-modules}. If an $\mathcal{A}$-module  $M$ is   finitely generated as  a $\mathbb{C}[\partial]$-module, then we call it a {\it finite} $\mathcal{A}$-module. The \emph{rank} of  $M$ is its rank as a $\mathbb{C}[\partial]$-module.  The  notions of a  submodule and an irreducible submodule are defined by using  $j$-th products in the obvious way.

An $\mathcal{A}$-module $M$ is said to be \emph{trivial} if $a_\lambda m=0$ for any $a\in\mathcal{A}$ and $m\in M$.
 Suppose that $M$ is a  finite  $\mathcal{A}$-module. Let $\text{Tor}M:=\{m\in M \mid f(\partial)m=0\  \text{for}\  \text{some nonzero}\  f(\partial) \in \mathbb{C}[\partial]\}$. Then $\text{Tor}(M)$ is a trivial  $\mathcal{A}$-submodule of $M$ by  \cite[Lemma 8.2]{DK}.  Thus any finite dimensional  $\mathcal{A}$-submodule of $M$ is trivial.
Let $\mathbb{C}_u=\mathbb{C}$ as a vector space for some $u\in\mathbb{C}$. Then  $\mathbb{C}_u$ becomes  a one-dimensional  $\mathcal{A}$-module with
   $\partial c=uc$ and $a_\lambda c=0$  for any $a\in \mathcal{A}$ and any  $c \in \mathbb{C}_u$.  Since any finite dimensional  $\mathcal{A}$-submodule of $M$ is trivial,  any irreducible finite dimensional  $\mathcal{A}$-module is in such form.  Hence we mainly focus  on studying  finite non-trivial $\mathcal{A}$-modules in the sequel.
\begin{example} {\it Regular module}. For any Lie conformal algebra $\mathcal{A}$, we can regard $\mathcal{A}$ as an $\mathcal{A}$-module with respect to the  $\lambda$-brackets of $\mathcal{A}$.  Thus $\text{Tor}(\mathcal{A})$ is in the center of $\mathcal{A}$.
 \end{example}
 Suppose $M$ is a finite $\mathcal{A}$-module. Consider the natural semiproduct Lie conformal algebra $\mathcal{A}\ltimes M$. Then $\text{Tor}(\mathcal{A})$ is in the center of $\mathcal{A}\ltimes M$. Thus  $\text{Tor}(\mathcal{A})$ acts trivially on $M$. Hence  we also focus on  finite Lie conformal algebras  which are free as $\mathbb{C}[\partial]$-modules in the sequel.

  The following  proposition is well known.

   \begin{proposition}(\cite[Theorem 8.2]{DK}){\label{t2.5}} Any non-trivial free rank one  $Vir$-module has the form $M_{a,b}=\mathbb{C}[\partial]v$, such that
       \[ L_\lambda v=(\partial+a\lambda+b)v,\]
        for $a,b \in \mathbb{C}$. Moreover, if $a \neq 0$, then  $M_{a,b}$ is irreducible and any non-trivial finite irreducible $Vir$-module is in  such form.  If $a=0$, then $M_{0,b}$ has the  unique  finite irreducible proper $Vir$-submodule $(\partial+b)M_{0,b}$, which is isomorphic to  $M_{1,b}$.
        \end{proposition}

 Finally, let us recall the (extended) annihilation Lie algebra of a Lie conformal algebra. The {\it annihilation Lie algebra} $\text{Lie}( \mathcal{A})^+$ of a Lie conformal algebra $ \mathcal{A}$ is the vector space   $\text{span}_\mathbb{C}\{a_{(n)}\mid a \in   \mathcal{A} , \  n\in \mathbb{Z}_+\}$ with  relations
     \[(\partial a)_{(n)}=-na_{(n-1)},\ (a+b)_{(n)}=a_{(n)}+b_{(n)},\ (ka)_{(n)}=ka_{(n)}, \]for  $a,b \in \mathcal{A}$ and  $k \in \mathbb{C}$,
and the  Lie brackets of  $\text{Lie}( \mathcal{A})^+$ are given by
  \[  [a_{(m)},b_{(n)}]= \sum\limits_{i\in \mathbb{Z}_+} \binom m i (a_{(i)}b)_{(m+n-i)}. \]
  The {\it extended annihilation Lie algebra} $\text{Lie}( \mathcal{A})^e$ is the Lie algebra $\mathbb{C}\partial \ltimes \text{Lie}( \mathcal{A})^+$ with the brackets $[\partial, a_{(n)}]=-na_{(n-1)}$.
  \begin{example} Suppose $\mathcal{A}=\mathbb{C}[\partial]L$ is the Virasoro Lie conformal algebra. Define $L_n=L_{(n+1)}$ for $n\geq -1$. Then $\text{Lie}(\mathcal{A})^+$ is a Lie algebra
  with basis $\{L_n,n\geq -1\}$ and  brackets as following:
    \[ [L_m,L_n]=(m-n)L_{m+n}, \]
 for   $\ m,n \in \mathbb{Z}_+$.
   One can see that  $ \text{Lie}(\mathcal{A})^+ \cong
  Vect(\mathbb{C})$ where $Vect(\mathbb{C})$ is the regular vector fields over $\mathbb{C}$.
    \end{example}

  A  $\text{Lie}( \mathcal{A})^e$-module $V$ is said to be {\it conformal} if for any $v \in V$ and $ a \in A$, there exists   $N(a,v) \in \mathbb{Z}_+$ such that $a_{(n)}v=0$ when $n>N(a,v)$. In fact, any $\mathcal{A}$-module determines a unique conformal $\text{Lie}( \mathcal{A})^e$-module and vice versa.

   \begin{proposition}{\label{t2}} (see \cite{CK})
   Suppose that $V$ is a finite $\mathcal{A}$-module.   For any $a \in A$, $v \in V$,  we write
     \begin{eqnarray}
     \begin{array}{ll}
     a_\lambda v=\sum_{j\in \mathbb{Z}_+}(a_{(j)}v)\frac{\lambda^j}{j!}.
     \end{array}
     \end{eqnarray}
     Then $V$ is a conformal $\text{Lie}( \mathcal{A})^e$-module with the action:
     \[a_{(n)}\cdot v=a_{(n)}v,\ \  \partial\cdot v=\partial v.\]
     Conversely, if $V$ is a conformal $\text{Lie}( \mathcal{A})^+$-module, then $V$ is also an $\mathcal{A}$-module with the action:
     \begin{eqnarray}
     \begin{array}{ll}
     a_\lambda v=\sum_{j\in \mathbb{Z}_+}(a_{(j)}\cdot v)\frac{\lambda^j}{j!}.
     \end{array}
     \end{eqnarray}
    In particular, $V$ is a finite irreducible $\mathcal{A}$-module if and only if $V$ is an  irreducible conformal  $\text{Lie}( \mathcal{A})^e$-module.
    \end{proposition}

     \section{Irreducible modules of free rank two Lie conformal algebras}

In  this section, we will study finite non-trivial irreducible modules of free rank two Lie conformal algebras. As an application, we prove that any  finite non-trivial irreducible modules of a finite Lie conformal algebra $M$ whose semisimple quotient is the Virasoro Lie conformal algebra, that is, $M/\text{Rad}(M)=Vir$, must be free rank of one as a $\mathbb{C}[\partial]$-module. Finally, the abelian extensions of some free rank two Lie conformal algebras by their rank one modules are also computed.

  We use $\mathcal{H}$ to denote a free  rank two Lie conformal algebra in the sequel. If $\mathcal{H}$ is semisimple, then $\mathcal{H}$ is the direct sum of two Virasoro Lie conformal algebras, where these two algebras are ideals of $\mathcal{H}$. As for the non-semisimple case,
 we have the following  proposition.
   \begin{proposition}\label{pp1}(\cite[ Theorem 2.21] {BCH}) If $\mathcal{H}$ is solvable, then there is a basis $\{A,B\}$ such that
     \begin{equation}{\label{312}}
      [B_\lambda B]=0,\ [A_\lambda B]=P_1(\partial,\lambda)B,\ [A_\lambda A]=Q_1(\partial,\lambda)B,
      \end{equation}
 where $P_1(\partial,\lambda), Q_1(\partial,\lambda)\in \mathbb{C}[\partial,\lambda]$.

  Suppose that $\mathcal{H}$ is  neither solvable nor semisimple. Then  there is a basis $\{A,B\}$  such that
      \begin{equation}\label{ee1}
      \begin{aligned}
      &[B_\lambda B]=0,\ [A_\lambda B]=\delta(\partial+a \lambda +b)B,\\
       &[A_\lambda A]=(\partial+2\lambda)A+Q(\partial,\lambda)B, \  \delta \in \{0,1\},
       \end{aligned}
      \end{equation}
where $a,b \in \mathbb{C}$ and   $Q(\partial,\lambda)$ is some polynomial depending on $a,b,\delta$.
More explicitly,  $Q(\partial,\lambda)=0$ if $b \neq 0$ or $\delta=0$.  Otherwise, we have

   \centerline{
   \begin{tabular}{|c|c|}
   \hline
   $a \in \mathbb{C}$  &  $Q(\partial, \lambda)$, $\beta, \gamma \in \mathbb{C}$, \\ \hline
      1 & $\beta(2\lambda+\partial)$ \\ \hline
      0 & $\beta(2\lambda+\partial)(\lambda^2+\lambda\partial)+\gamma(2\lambda+\partial)\partial$ \\ \hline
      -1 & $\beta(2\lambda+\partial)\partial^2+\gamma(2\lambda+\partial)(\lambda^2+\lambda\partial)\partial$ \\ \hline
      -4  & $\beta(2\lambda+\partial)(\lambda^2+\lambda\partial)^3$ \\ \hline
      -6  & $\beta(2\lambda+\partial)[11(\lambda^2+\lambda\partial)^4+2(\lambda^2+\lambda\partial)^3\partial^2]$\\  \hline
      \end{tabular}
      }
      \end{proposition}
   \begin{remark} Let  $\mathcal{H}$ be a  free rank two Lie conformal algebra. Then $Ado\ Theorem$ holds for $\mathcal{H}$, that is, $\mathcal{H}$ has a finite  faithful module. \cite[Theorem 3]{PK} states that any Lie conformal algebra in the form $\mathcal{A}=\mathcal{S} \ltimes \mathcal{R}$, where  $\mathcal{S}$ is a finite semisimple Lie conformal algebra or $\mathcal{S}=0$ and $\mathcal{R}$ is the radical ideal of $\mathcal{A}$,  has a finite  faithful module.
    Thus it is clear when $\mathcal{H}$ is semisimple or solvable.  Suppose $\mathcal{H}=\mathbb{C}[\partial]A \oplus \mathbb{C}[\partial]B$ is a Lie conformal algebra defined in  (\ref{ee1}). If $Q(\partial,\lambda)$ is trivial, then by \cite[Theorem 3]{PK}, $\mathcal{H}$ has a finite  faithful module.   If $Q(\partial,\lambda)$ is non-trivial, we claim that   $\mathcal{H}$ is centerless.
    Suppose that $f(\partial)A+g(\partial)B \in \mathcal{H}$ is in the center of $\mathcal{H}$.  Then
        \[  [(f(\partial)A+g(\partial)B)_\lambda A]= f(-\lambda)(\partial+2\lambda)A-g(-\lambda)(\partial-a \partial -a\lambda+b)B=0.  \]
      Thus, $f(\partial)=g(\partial)=0$. Hence the regular module is a finite faithful module of $\mathcal{H}$ in this case.
      \end{remark}

    \begin{proposition}{\label{p3}} Set $\mathcal{H}=\mathbb{C}[\partial]A \oplus \mathbb{C}[\partial]B$ as direct sum of ideals, that is $[A_\lambda B]=0$, where $\mathbb{C}[\partial]A$ is the Virasoro Lie conformal algebra. Let $V$ be a  finite  $\mathcal{H}$-module such that $A$ acts non-trivially on $V$.
  Then  there exists $u \in V$ such that  $A_\lambda u=(\partial+\alpha\lambda+\beta)u$ for some $\alpha$, $\beta \in \mathbb{C}$ and $B_\lambda u=0$.
    \end{proposition}
     \begin{proof}
   Consider $V$ as a  $\mathbb{C}[\partial]A$-module. Thus by Proposition \ref{t2.5}, there exists $u \in V$ such that  $A_\lambda u=(\partial+\alpha\lambda+\beta)u$ for some $\alpha$, $\beta \in \mathbb{C}$. Assume that  $B_\lambda u=\sum\limits_{i \in \mathbb{Z}_+} u_i\lambda^i$ where $u_i \in V$.
 Then
       \begin{equation}\label{elast}
       A_{(0)}(B_\lambda u)=\sum\limits_{i \in \mathbb{Z}_+}(A_{(0)}\cdot u_i)\lambda^i=  B_\lambda(A_{(0)}\cdot u)=(\partial+\lambda+\beta)\sum\limits_{i \in \mathbb{Z}_+} u_i\lambda^i.
       \end{equation}
     By comparing the coefficients of $\lambda$ of Equation (\ref{elast}), we have $B_\lambda u=0$ which implies that $\mathbb{C}[\partial]u$ is an $\mathcal{H}$-module.
\end{proof}

\begin{proposition}{\label{p1}} Suppose that $\mathcal{H}$ is the Lie conformal algebra defined in   (\ref{ee1}) with  $\delta=1$ and  $V$ is a  non-trivial finite  $\mathcal{H}$-module. Then there  exists a nonzero element $u\in V$ such that $\mathbb{C}[\partial]u$ is an $\mathcal{H}$-submodule of rank one as a $\mathbb{C}[\partial]$-module. In particular, if $V$ is irreducible, then $V=\mathbb{C}[\partial]u$.
\end{proposition}
 \begin{proof}  If $Q(\partial,\lambda)$ is not a constant, let us use   $d$  to denote  the total degree of $Q(\partial,\lambda)$.  Otherwise, we let $d=1$.

  Then the corresponding  extended annihilation Lie algebra $\text{Lie}(\mathcal{H})^e=\mathbb{C}\partial \oplus \sum_{n \geq 0}\mathbb{C}A_{(n)} \oplus  \sum_{n \geq 0}\mathbb{C}B_{(n)}$, whose brackets  are determined by
      \begin{equation}{\label{e1}}
   \begin{aligned}
    &[A_{(m)},A_{(n)}]=(m-n)A_{(m+n-1)}+\sum \limits_{i=0}^d a_i B_{(m+n-i)}, \\
    &[A_{(m)},B_{(n)}]=((a-1)m-n)B_{(n+m-1)}+ bB_{(n+m)},    \\
   &[\partial,A_{(n)}]=-nA_{(n-1)}, \ \ [\partial,B_{(n)}]=-nB_{(n-1)},\ \ [B_{(m)},B_{(n)}]=0,
   \end{aligned}
   \end{equation}
   where   $a_i \in \mathbb{C}$ for $0\leq i \leq d$ depending on   $Q(\partial,\lambda)$.

Define $\mathcal{L}_p = \text{span}_\mathbb{C}\{A_{(s+d)}
, B_{(s)}|s \geq p \}$ for $p \geq 0$. Since $V$
is a non-trivial conformal $\mathcal{L}_{0}$-module, by Proposition $\ref{t2}$, there exists a minimal
integer $N \geq 0$ such that $U := \{v \in V | \mathcal{L}_N\cdot v=0\} \neq \{0\}$.
 In addition, $U$
is  finite dimensional by Lemma 14.4 of \cite{BDK}. Let $\mathcal{N}=\mathbb{C}(\partial-A_{(0)}) \oplus \mathcal{L}_0 \oplus \sum\limits_{i=1}^{d-1}\mathbb{C}A_{(i)}$.  Then $\mathcal{N}$ is a subalgebra of $\text{Lie}(\mathcal{H})^e$.
Moreover, we have $[\mathcal{N},\mathcal{L}_p]\subset \mathcal{L}_p$ for any $p \geq 0$. Thus $U \subset V$ is a nonzero non-trivial finite dimensional  $\mathcal{N}/\mathcal{L}_{N}$-module.
      Note that  \[\mathcal{N}^{(1)}:=[\mathcal{N},\mathcal{N}]\subset \mathcal{L}_0 \oplus \sum\limits_{i=1}^{d-1}\mathbb{C}A_{(i)}.\]
      Define inductively $\mathcal{N}^{(i+1)}:=[ \mathcal{N}^{(i)},\mathcal{N}^{(i)}]$. Then
    $  \mathcal{N}^{(d)} \subset \mathcal{L}_0$. Thus $\mathcal{N}/\mathcal{L}_N$ is a solvable Lie algebra. Therefore, by {\it Lie's Theorem}, there exists a nonzero element  $u\in U$ and a linear function $\xi$ on $\mathcal{N}$ such that $A_{(0)}\cdot u=(\partial+\beta)u$ for some $\beta \in \mathbb{C}$, $A_{(i)}\cdot u=\xi(A_{(i)})u$ for $i>0$ and                                      $B_{(j)} \cdot u=\xi(B_{(j)})u$ for $j\geq 0$.
   Thus  $A_\lambda u=(\partial+p(\lambda))u$ and $B_\lambda u=f(\lambda)u$ for some $p(\lambda)$, $f(\lambda) \in \mathbb{C}[\lambda]$. Hence, $\mathbb{C}[\partial]u$ is a free rank one  $\mathcal{H}$-submodule of $V$.
     \end{proof}

 By  Propositions {\ref{p3}} and Proposition {\ref{p1}}, for any free rank two Lie conformal algebra $\mathcal{H}$, it is easy to see that any  finite non-trivial irreducible module of $\mathcal{H}$ must be free of rank one.

 More explicitly, we have the following theorem as the main results of this paper.

 \begin{theorem}{\label{tf}}  Suppose that $\mathcal{H}=\mathbb{C}[\partial]A\oplus \mathbb{C}[\partial]B$ is a Lie conformal algebra of rank two. Then any non-trivial finite irreducible $\mathcal{H}$-module is free of rank one. Moreover,  if  $V=\mathbb{C}[\partial]v$ is a non-trivial irreducible $\mathcal{H}$-module, then  the action of $\mathcal{H}$ on $V$ has to be  one of the  following cases:
     \begin{enumerate}[fullwidth,itemindent=2em]
       \item[(i)]If $\mathcal{H}$ is solvable with the relations (\ref{312}), then we have
       $ A_\lambda v=\phi_A(\lambda)v$, $B_\lambda v=\phi_B(\lambda)v$,
where $\phi_A(\lambda)$, $\phi_B(\lambda)$ are not zero  simultaneously.
Moreover, $\phi_B(\lambda) \neq 0 $ only if $P_1(\partial,\lambda)=Q_1(\partial,\lambda)=0$.
       \item[(ii)]  Suppose that $\mathcal{H}$ is the Lie conformal algebra defined in   (\ref{ee1}) with  $\delta=1$. Then
       \[ A_\lambda v=(\partial+\alpha\lambda+\beta),\ B_\lambda v=\gamma v,\]
where $\alpha, \beta, \gamma \in \mathbb{C}$ such that $\gamma \neq 0$ only if $a=1$, $b=0$ and $Q(\partial,\lambda)=0$. Further, if $\gamma=0$, then $\alpha \neq 0$.
       \item[(iii)]Suppose that $\mathcal{H}$ is the Lie conformal algebra defined in   (\ref{ee1}) with  $\delta=0$, then either \[ A_\lambda v=(\partial+\alpha\lambda+\beta)v,\ B_\lambda v=0, \  \text{for}\ \text{some}  \   \beta,\  0 \neq\alpha \in \mathbb{C},\]
       or
       \[ A_\lambda v=0,\ B_\lambda v=\phi(\lambda)v, \  \text{for}\ \text{some}\ \text{nonzero}\  \phi(\lambda)\in \mathbb{C}[\lambda].\]
        \item[(iv)] If $\mathcal{H}=\mathbb{C}[\partial]A \oplus \mathbb{C}[\partial]B$ is a direct sum of two   Virasoro Lie conformal algebras with $[A_\lambda B]=0$,  then  either
        \[ A_\lambda v=(\partial+\alpha_1\lambda+\beta_1)v,\ B_\lambda v=0, \  \text{for}\ \text{some}  \   \beta_1,\  0 \neq\alpha_1 \in \mathbb{C},\]
       or
       \[ A_\lambda v=0,\ B_\lambda v=(\partial+\alpha_2\lambda+\beta_2)v, \    \text{for}\ \text{some}  \   \beta_2,\  0 \neq\alpha_2 \in \mathbb{C}.\]
        \end{enumerate}
\end{theorem}
\begin{proof}
\begin{enumerate}[fullwidth,itemindent=2em]
    \item[($i$)] It is clear from the definition and  Theorem 8.4 of \cite{DK}.
   \item[($ii$)] If $\mathcal{H}$ is neither solvable nor semisimple, then by Proposition {\ref{p1}}, $A_\lambda v=(\partial+p(\lambda))v$ and $B_\lambda v=f(\lambda)v$ for some $p(\lambda)$, $f(\lambda) \in \mathbb{C}[\lambda]$.
    Plugging these into the equations:
        \begin{equation}
        \begin{aligned}
        &[A_\lambda B]_{\lambda+\mu}v=A_\lambda B_\mu v- B_\mu A_\lambda v,\\
        &[A_\lambda A]_{\lambda+\mu}v=A_\lambda A_\mu v- A_\mu A_\lambda v,
        \end{aligned}
\end{equation}
   we have
        \begin{equation}{\label{e22}}
        \begin{array}{l}
        f(\lambda+\mu)(-\lambda-\mu+a\lambda+b)=-\mu f(\mu),
        \end{array}
\end{equation}
 \begin{equation}{\label{e23}}
        \begin{array}{l}
       (\lambda-\mu)p(\lambda+\mu)+Q(-\lambda-\mu,\lambda) f(\lambda+\mu)=\lambda p(\lambda)-\mu p(\mu).
        \end{array}
\end{equation}
    If $f(\lambda)\neq 0$, Equation (\ref{e22}) forces that  $a=1$, $b=0$ and $f(\lambda)=\gamma$ for some nonzero  $\gamma \in \mathbb{C}$. In this case, $Q(\lambda,\mu)=0$. Otherwise,  Equation (\ref{e23}) implies that $Q(\lambda,\mu)$  is a 2-coboundary  in $C^2(Vir,M_{1,0})$, which contradicts to the construction of $Q(\lambda,\mu)$. Thus, in any case,  $Q(-\lambda-\mu,\lambda) f(\lambda+\mu)=0$. The remainder of (ii)  is obvious from Proposition \ref{t2.5}.
   \item[($iii$) and ($iv$)] are immediate results by   Proposition {\ref{p3}}.

   \end{enumerate}
   \end{proof}
   As an application, we have the following theorem.
 \begin{theorem}{\label{bcc29}}
 Let $M$ be a free finite Lie conformal algebra such that $M/\text{Rad}(M) = Vir$.  Then any non-trivial finite $M$-module
$V$ must contain a free rank one $M$-submodule.
\end{theorem}
\begin{proof}
Note that $\text{Rad}(M)$ is the radical of $M$. Then by the \emph{conformal Lie's Theorem},
a nonzero eigenspace $V_\phi$ exists, where $V_\phi=\{v \in V| m_\lambda v=(\phi(m)|_{\partial=-\lambda})v, \ \forall m\in \text{Rad}(M) \}$ for some $\phi \in \text{Hom}_{\mathbb{C}[\partial]}(\text{Rad}(M), \mathbb{C}[\partial])$.
By Proposition 14.1 of \cite{BDK}, $V_\phi$ spans an $M$-module, which
is free by Lemma 4.1 of \cite{BDK}.
Hence, we may assume that $V=\mathbb{C}[\partial]\otimes_{\mathbb{C}} V_\phi$.
 Note that $V_\phi$ is a finite dimensional vector space according to that $V$ is a finite  $M$-module.

Thus we may assume that $V=\mathbb{C}[\partial]\otimes_{\mathbb{C}} V_\phi$  for some $\phi \in \text{Hom}_{\mathbb{C}[\partial]}(\text{Rad}(M), \mathbb{C}[\partial])$.
Since $Vir$ is free as a $\mathbb{C}[\partial]$-module, $M =\mathbb{C}[\partial]L \oplus \text{Rad}(M)$ where
\[ [L_\lambda L] = (\partial + 2\lambda)L \ \  \text{mod}\  \text{Rad}(M). \]
Suppose $\{ M_i|1 \leq i \leq k \}$ is a $\mathbb{C}[\partial]$-basis of $\text{Rad}(M)$.
  If $L$ or $\text{Rad}(M)$  acts  trivially on $V$, then the theorem is obvious. Hence let us  assume that
 both $L$ and $M_1$ act non-trivially on $V$. Therefore,  $\phi(M_1)\neq 0$.

Define $\text{ker}(V)=\{a\in M|a_\lambda u=0, \forall u \in
 V\}$. Then $\text{ker}(V)$ is an ideal of $M$. We are
going to prove that $(\mathbb{C}[\partial]L \oplus \mathbb{C}[\partial]M_1) \cap \text{ker}V=\{0\}$.
Assume that there exist  $f(\partial)$ and $g(\partial) \in \mathbb{C}[\partial]$ such that $(f(\partial)L+g(\partial)M_1)_\lambda v=0$ for any $v \in  V_\phi$. On one hand, $f(\partial)=0$
 implies that $g(\partial)=0$. On the other hand, if $f(\partial)\neq 0$, $L_\lambda v\in V_\phi[\lambda]$ which implies that $V_\phi$ is a finite dimensional $M$-module. In this case, $M$ acts trivially on $V_\phi$ which contradicts to our assumption.     As a consequence,  $\text{rank}(M/\text{ker}(V))\geq 2$ and $\text{rank}(\text{ker}(V))\leq k-1$.

 Define $T_i=\phi(M_i)M_1-\phi(M_1)M_i$ for $2\geq i \geq k$. Let $T$ be the $\mathbb{C}[\partial]$-module generated by $\{T_2,T_3, \cdots, T_k \}$. Then $T \subset \text{ker}(V)$ and $\text{rank}(T)=k-1$. Thus $k-1\leq \text{rank}(T)\leq \text{rank}(\text{ker}(V))\leq k-1$. Hence $\text{rank}(M/\text{ker}(V))=2$. Then, this reduces to the study of finite irreducible modules of rank two Lie conformal algebras.
 Now by Theorem {\ref{tf}}, we complete the proof.
  \end{proof}
  \begin{remark}
Since the torsion part of a finite Lie conformal algebra  acts trivially on any module of this Lie conformal algebra, by Theorem \ref{bcc29}, any finite non-trivial irreducible modules of  a finite Lie conformal algebra $M$ satisfying $M/\text{Rad}(M)=Vir$ is free of rank one. Therefore, any finite non-trivial irreducible modules
of Lie conformal algebras such as $\mathcal{W}(a,b)$, Schr\"odinger-Virasoro type Lie  conformal algebras in \cite{LHW} is free of rank one. This result provides a unified method to determine finite irreducible modules of such Lie conformal algebras.
\end{remark}

We finish this paper by computing the  abelian extensions of some free rank two Lie conformal algebras by their non-trivial  free rank one modules. As in Theorem 3.1 of \cite{BKV}, these  abelian extensions can be described in terms of the second cohomology groups.

\begin{proposition}\label{ppp1} Let $\mathcal{H}=\mathbb{C}[\partial]A \oplus \mathbb{C}[\partial]B$ be a  rank two Lie conformal algebra with $\lambda$-brackets:
 \[[A_\lambda A]=(\partial+2\lambda)A,\ \  [A_\lambda B]=0,\ \ [B_\lambda B]=\delta (\partial+2\lambda)B,\ \ \delta \in \{0,1\}.\] Suppose $M_{a,b}=\mathbb{C}[\partial]v$ is a rank one $\mathcal{H}$-module defined as following:
    \[ A_\lambda v=(\partial+a\lambda+b)v,\ \ B_\lambda v=0, \]
    for some $a,b \in \mathbb{C}$. Then  $H^2(\mathcal{H},M_{a,b})\cong H^2(\mathbb{C}[\partial]A,M_{a,b})$.
\end{proposition}
\begin{proof}
Suppose $\mathcal{E}$ is  an abelian extension of  $\mathcal{H}$ by $M_{a,b}$.  Then  the nontrivial $\lambda$-brackets of $\mathcal{E}$  are given by
     \[  [A_\lambda A]=(\partial+2\lambda)A+Q_1(\partial,\lambda)v, \ \ [A_\lambda v]=(\partial+a\lambda+b)v, \ \ \]
       \[ [A_\lambda B]=Q_2(\partial,\lambda)v,\ \ [B_\lambda B]=\delta (\partial+2\lambda)B+Q_3(\partial,\lambda)v, \delta \in \{0,1\}, \ \  \]
for some $Q_1(\partial,\lambda)$, $Q_2(\partial,\lambda)$, $Q_3(\partial,\lambda)\in \mathbb{C}[\partial,\lambda]$.
The Jacobi-identites
      \[[A_\lambda [A_\mu B]]=[[A_\lambda A]_{\lambda+\mu}B]+[A_\mu[A_\lambda B]],\]
      \[[A_\lambda [B_\mu B]]=[[A_\lambda B]_{\lambda+\mu}B]+[B_\mu[A_\lambda B]],\]
 provide
\begin{equation}{\label{yyf}}
   (\lambda-\mu)Q_2(\partial,\lambda+\mu)=Q_2(\partial+\lambda,\mu)(\partial+a\lambda+b)-Q_2(\partial+\mu, \lambda)(\partial+a\mu+b)
   \end{equation}
and
   \begin{equation}{\label{yyg}}
   \delta (\partial+\lambda+2\mu)Q_2(\partial,\lambda)+Q_3(\partial+\lambda,\mu)(\partial+a\lambda+b)=0.
   \end{equation}
 Applying $\lambda=0$ in (\ref{yyg}),
   \[Q_3(\partial,\mu)=-\delta (\partial+2\mu)f(\partial)\]
   where $f(\partial)=\frac{Q_2(\partial,0)}{\partial+b}$.
   Similarly, applying $\mu=0$ in (\ref{yyf}), we get
   \begin{equation}
       Q_2(\partial,\lambda)=(\partial+a\lambda+b)f(\partial+\lambda).
       \end{equation}
   Define $\widetilde{B}=B-f(\partial)v$. Thus $\mathcal{E}=\mathbb{C}[\partial]A \oplus \mathbb{C}[\partial]\widetilde{B} \oplus M_{a,b}$ as direct sum of $\mathbb{C}[\partial]$-modules. Moreover, it is straightforward to check that $[A_\lambda \widetilde{B}]=0$ and $[\widetilde{B}_\lambda\widetilde{B}]=\delta (\partial+2\lambda) \widetilde{B}$. Thus, by a suitable choice of basis of $\mathcal{E}$ as a free $\mathbb{C}[\partial]$-module, we can take $Q_2(\partial,\lambda)=Q_3(\partial,\lambda)=0$. This completes the proof.
\end{proof}

\begin{proposition}\label{ppp2} Let $\mathcal{H}=\mathbb{C}[\partial]A \oplus \mathbb{C}[\partial]B$ be a  rank two Lie conformal algebra with $\lambda$-brackets:
 \[[A_\lambda A]=(\partial+2\lambda)A,\ \  [A_\lambda B]=0,\ \ [B_\lambda B]=0.\] Suppose $V=\mathbb{C}[\partial]v$ is a rank one $\mathcal{H}$-module defined as following:
    \[ A_\lambda v=0,\ \ B_\lambda v=\phi(\lambda)v,\ \ \phi(\lambda)\in \mathbb{C}[\partial]. \]
  Then  $H^2(\mathcal{H},V)\cong H^2(\mathbb{C}[\partial]B,V)$.
\end{proposition}
\begin{proof}
Suppose $\mathcal{E}$ is  an abelian  extension of  $\mathcal{H}$ by $V$.  Then  the nontrivial $\lambda$-brackets of $\mathcal{E}$ are given by
     \[  [A_\lambda A]=(\partial+2\lambda)A+Q_1(\partial,\lambda)v, \ \ [B_\lambda v]=\phi(\lambda)v, \]
       \[ [B_\lambda A]=Q_2(\partial,\lambda)v,\ \ [B_\lambda B]=Q_3(\partial,\lambda)v, \]
for some $Q_1(\partial,\lambda)$, $Q_2(\partial,\lambda)$, $Q_3(\partial,\lambda)\in \mathbb{C}[\partial,\lambda]$.
The Jacobi-identites
      \[[A_\lambda [A_\mu A]]=[[A_\lambda A]_{\lambda+\mu}A]+[A_\mu[A_\lambda A]],\]
      \[[B_\lambda [A_\mu A]]=[[B_\lambda A]_{\lambda+\mu}A]+[A_\mu[B_\lambda A]],\]
 provide
\begin{equation}{\label{zju1}}
   (\lambda-\mu)Q_1(\partial,\lambda+\mu)=(\partial+\lambda+2\mu)Q_1(\partial,\lambda)-(\partial+2\lambda+\mu)Q_1(\partial,\mu)
   \end{equation}
and
   \begin{equation}{\label{zju2}}
   (\partial+\lambda+2\mu)Q_2(\partial,\lambda)+\phi(\lambda)Q_1(\partial+\lambda,\mu)=0.
   \end{equation}
Setting $\mu=0$ in both  (\ref{zju1}) and  (\ref{zju2}),  we get
      \[ Q_1(\partial,\lambda)=(\partial+2\lambda)f(\partial),\]
   where $f(\partial)=\frac{Q_1(\partial,0)}{\partial}$,
and
    \[Q_2(\partial,\lambda)=-\phi(\lambda)f(\partial+\lambda).\]
   Define $\widetilde{A}=A+f(\partial)v$. Thus $\mathcal{E}=\mathbb{C}[\partial]\widetilde{A} \oplus \mathbb{C}[\partial]{B} \oplus V$ as direct sum of $\mathbb{C}[\partial]$-modules. Moreover, it is straightforward to check that $[B_\lambda \widetilde{A}]=0$ and $[\widetilde{A}_\lambda\widetilde{A}]=(\partial+2\lambda) \widetilde{A}$. Thus, by a suitable choice of basis of $\mathcal{E}$ as a free $\mathbb{C}[\partial]$-module, we can take $Q_1(\partial,\lambda)=Q_2(\partial,\lambda)=0$. This completes the proof.
\end{proof}

\section*{Acknowledgements}
We wish to thank the referee for suggesting the
problem on whether any finite non-trivial irreducible  module of a finite Lie conformal algebra $M$ with $M/\text{Rad}(M)=Vir$ is free of rank one (from which we
obtain Theorem \ref{bcc29}) and computing abelian extensions of some rank two Lie conformal algebras by its non-trivial finite irreducible modules (from which we
obtain Proposition \ref{ppp1} and Proposition \ref{ppp2}).


\begin{thebibliography}{XTW}
\bibitem [BKV]{BKV} B. Bakalov, A. Voronov, V. Kac, Cohomology of comformal algebras,
\emph{Commun. Math. Phys}, {\bf{200}}(1999), 561-598.
        \bibitem[BCH]{BCH}R. Biswal, A. Chakhar, X. He, Classification of rank two Lie conformal algebras, arXiv: 1712.05478.
        \bibitem[BDK]{BDK}B. Bakalov, A.  D'Andrea, V. Kac, Theory of finite pseudoalgebras,   \emph{Adv. Math}, {\bf{162}}(2001),  1-140.
        \bibitem[CK]{CK} S. Cheng and V. Kac,  Conformal modules, \emph{Asian J. Math}. {\bf{1}}(1997), 181-193.
         \bibitem[DK]{DK} A. D' Andrea, V.  Kac,  Structure theory of finite conformal algebras, \emph{Sel. Math}, {\bf{4}}(1998), 377-418.
         \bibitem[HF]{HF} Y. Hong, F. Li, Virasoro-type Lie conformal algebras of rank 2, \emph{Chinese Annal. Math. A}, {\bf{39}}(2018), 15-32.
          \bibitem [K] {K} V. Kac, \text{Vertex algebras for beginners}. Univ. Lect. Series 10, AMS(1996),  Second edition 1998.
              \bibitem[K1]{K1} V. Kac, Classification of rank 2 Lie conformal algebras, \emph{In communication}, 2011.
\bibitem[LHW]{LHW}L. Luo, Y. Hong, Z. Wu, Finite irreducible modules of Lie conformal algebra $W(a,b)$
and some Schr\"odinger-Virasoro type Lie  conformal algebras, {\emph{Int. J. Math}}, {\bf{30}}(2019), 1950026.
\bibitem[PK]{PK}P. Kolesnikov, On the Ado Theorem for finite Lie conformal algebras with Levi decomposition. \emph{J. Algebra Appl.}, {\bf{15}}(2016), 1650130.
\bibitem[SXY]{SXY} Y Su, C Xia, L Yuan,  Classification of finite irreducible conformal modules over a class of Lie conformal algebras of Block type. \emph{J. Algebra}, {\bf{499}}(2018), 321-336.
\bibitem[W]{W} Z. Wu,  Schr\"odinger-Virasoro Lie $H$-Pseudoalgebras, arxiv.1809.03295.
\bibitem[W1]{W1} Z. Wu, Leibniz Conformal Algebras of Rank Two, Acta. Math. Sin.-English Ser. 36, 109-120 (2020).
\bibitem[WY]{WY} H. Wu, L. Yuan, Classification of finite irreducible conformal modules over some Lie conformal algebras related to the Virasoro conformal algebra,  \emph{J. Math. Phys}, {\bf{58}}(2017), 041701.
\end{thebibliography}
\end{document}